\documentclass[12pt, reqno, a4paper]{amsart}

\usepackage{ amssymb, amsmath, enumerate, amsfonts, amsthm, mathrsfs, url, bm, mathtools}

\numberwithin{equation}{section}

\usepackage{xcolor}  	
\usepackage[backref]{hyperref}
\hypersetup{
	colorlinks,
    linkcolor={blue!60!black},
    citecolor={blue!60!black},
    urlcolor={red!60!black}
}

\usepackage{color}

\usepackage[margin=1.25in]{geometry}

\RequirePackage{doi}

\usepackage[square,sort,comma,numbers]{natbib}
 \newcommand{\set}[1]{\left\{#1\right\}}

\newcommand{\abs}[1]{\left| #1\right|}
\newcommand{\bigabs}[1]{\bigl| #1 \bigr|}
\newcommand{\Bigabs}[1]{\Bigl| #1 \Bigr|}

\newcommand{\sqbrac}[1]{\left[ #1 \right]}

\newcommand{\ceil}[1]{\left\lceil #1 \right\rceil}

\newcommand{\floor}[1]{\left\lfloor #1 \right\rfloor}
\newcommand{\brac}[1]{\left( #1 \right)}

\newcommand{\Bigbrac}[1]{\Bigl( #1 \Bigr)}
\newcommand{\biggbrac}[1]{\biggl( #1 \biggr)}

\newcommand{\norm}[1]{\left\| #1\right\|}
\newcommand{\bignorm}[1]{\big\| #1 \big\|}

\newcommand{\ang}[1]{\left\langle#1\right\rangle}

\newcommand{\recip}[1]{\frac{1}{#1}}
\newcommand{\trecip}[1]{\tfrac{1}{#1}}

\newcommand{\N}{\mathbb{N}}
\newcommand{\Z}{\mathbb{Z}}

\newcommand{\R}{\mathbb{R}}
\newcommand{\C}{\mathbb{C}}
\newcommand{\T}{\mathbb{T}}

\newcommand{\E}{\mathbb{E}}

\newcommand{\lcm}{\mathrm{lcm}}

\newcommand{\supp}{\mathrm{supp}}

\makeatletter
\let\@@pmod\pmod
\DeclareRobustCommand{\pmod}{\@ifstar\@pmods\@@pmod}
\def\@pmods#1{\mkern4mu({\operator@font mod}\mkern 6mu#1)}
\makeatother

\renewcommand{\leq}{\leqslant}
\renewcommand{\geq}{\geqslant}

\newtheorem{theorem}{Theorem}[section]

\newtheorem{corollary}[theorem]{Corollary}

\newtheorem{lemma}[theorem]{Lemma}

\theoremstyle{definition}
\newtheorem{definition}[theorem]{Definition}

\numberwithin{theorem}{section}

\title[Polylog nonlinear Roth]{A polylogarithmic bound in the nonlinear Roth theorem} 

\author{Sarah Peluse}
\address{Mathematical Institute\\
University of Oxford\\ %Oxford\\  
UK}
\email{sarah.peluse@maths.ox.ac.uk}

\author{Sean Prendiville}
\address{Department of Mathematics and Statistics\\
Lancaster University\\
UK}
\email{s.prendiville@lancaster.ac.uk}

\begin{document}

\begin{abstract}
We show that sets of integers lacking the configuration $x$, $x+y$, $x+y^2$ have at most polylogarithmic density.   
\end{abstract}

\maketitle

\setcounter{tocdepth}{1}
\tableofcontents

\maketitle

\section{Introduction}\label{introduction}

\subsection{Density bound}
In \cite{PelusePrendivilleQuantitative} the authors obtained, for the first time, an  effective bound for subsets of $\set{1, \dots, N}$ lacking the nonlinear Roth configuration $x$, $x+y$, $x+y^2$. There it was established that such sets have cardinality at most $O(N/(\log\log N)^c)$, where $c> 0$ is an absolute constant. The key breakthrough of~\cite{PelusePrendivilleQuantitative} was a ``local $U^1$-control'' result, from which a bound for sets lacking the nonlinear Roth configuration follows via standard methods.  Here, we combine this local $U^1$-control result with a more sophisticated argument to remove a logarithm from the bound of~\cite{PelusePrendivilleQuantitative}.

\begin{theorem}[Density bound]\label{main}
There exists an absolute constant $c > 0$ such that the following holds. 
Suppose that $A \subset \set{1,\dots, N}$ lacks configurations of the form
\begin{equation}\label{main config}
x,\ x+y, \ x+y^2 \qquad (y \neq 0).
\end{equation}
Then%\footnote{For our conventions regarding asymptotic notation, see \S\ref{notation}.} 
$$
|A| = O\left( N/(\log N)^{c}\right).
$$
\end{theorem}

A careful analysis shows that the exponent $c = 2^{-150}$ is permissible, where 150 represents the combined number of times we utilise the Cauchy--Schwarz inequality in~\cite{PelusePrendivilleQuantitative} and this paper

%\begin{remark}\label{rmk1.2}
%We did not try to optimize the exponent of $\log(N)$ in Theorem~\ref{main}, preferring instead to keep explicit quantities appearing in the proof as simple as possible.
%\end{remark}

\subsection{Major arc correlation}
The techniques which yield Theorem \ref{main} also allow us to show, in a quantitatively effective manner, that the major arc Fourier coefficients of a set  determine how many nonlinear Roth configurations \eqref{main config} the set contains.

\begin{theorem}[Major-arc control]\label{global control}
Let $\delta > 0$ and $f, g, h : \Z \to \C$ be 1-bounded functions with support in $\set{1, \dots, N}$.  Suppose that  
\[
\abs{\sum_{x\in\Z}\sum_{ y\in \N} f(x) g(x+y)h(x+y^2) }\geq \delta N^{3/2}.
\]
Then either $N \ll \delta^{-O(1)}$, or there is a frequency $\alpha \in \R$ and a positive integer $q \ll \delta^{-O(1)}$ such that\footnote{Here $\norm{\cdot}$ denotes the distance to the nearest integer, and $e(\alpha) := e^{2\pi i\alpha}$. For our conventions regarding asymptotic notation see \S\ref{notation}.} $\norm{q\alpha} \ll \delta^{-O(1)}/N$ and
$$
\abs{\sum_{x\in \Z} h(x)e(\alpha x) } \gg \delta^{O(1)} N.
$$
\end{theorem}
In the nomenclature of \cite{TaoObstructions}, the major arc linear phases are the only obstructions to uniformity for the nonlinear Roth configuration. We emphasise that Theorem \ref{global control} is not used in the proof of Theorem \ref{main config}.

The major arc Fourier coefficients of a subset of $\{1,\dots,N\}$ essentially measure its distribution in arithmetic progressions of common difference $\ll 1$ and length $\gg N$. To illustrate this, the following definition is useful.

\begin{definition}[Local function]\label{local function def}
We call a function $\phi : \Z \to \C$ a \emph{local function of resolution $M$ and modulus $q$} if there exists a partition of $\Z$ into intervals of length $M$ such that $\phi$ is constant on the intersection of every such interval with every congruence class mod $q$.
\end{definition}

\begin{corollary}[Local control of the nonlinear term]\label{global control cor}
Let $\delta > 0$ and $f, g, h : \Z \to \C$ be 1-bounded functions with support in $\set{1, \dots, N}$.  Suppose that  
\[
\abs{\sum_{x\in\Z}\sum_{ y\in \N} f(x) g(x+y)h(x+y^2) }\geq \delta N^{3/2}.
\]
Then either $N \ll \delta^{-O(1)}$, or there is a 1-bounded local function $\phi $ of resolution $M \gg \delta^{O(1)} N$ and modulus $q \ll \delta^{-O(1)}$ such that
$$
\abs{\sum_{x\in \Z} h(x)\phi(x) } \gg \delta^{O(1)} N.
$$
\end{corollary}

One cannot hope to prove that the functions $f$ and $g$ above also correlate globally with local functions, as the following example illustrates.  For any positive integers $x_1, x_2 \leq N^{1/2}$, set
$$
f\brac{x_1 + (x_2-1)\floor{N^{1/2}}} = \begin{cases} 
		1 & \text{ if } x_2 \equiv 0 \pmod 4,\\
		0 & \text{ if } x_2 \equiv 1 \pmod 4,\\ 
		-1 & \text{ if } x_2 \equiv 2 \pmod 4,\\
		0 & \text{ if } x_2 \equiv 3 \pmod 4; \end{cases}
$$
and set $f(x) = 0$ everywhere else.
Taking $g := f$ and $h:=1_{\{1,\dots, N\}}$, one can check that either $N \ll 1$ or
$$
\sum_{x\in \Z}\sum_{y\in \N} f(x)g(x+y) h(x+y^2) \gg N^{3/2}.
$$
However, for any arithmetic progression $P\subset\{1, \dots, N\}$, we have
$$
\abs{\sum_{x \in P} f(x)} \ll N^{1/2}.
$$
Hence, for any 1-bounded local function $\phi$ of resolution $\geq \delta N$ and modulus $\leq \delta^{-1}$, the triangle inequality gives the discorrelation 
$$
\abs{\sum_{x \in \Z } f(x) \phi(x)} \ll \delta^{-2} N^{1/2}.
$$

This example is a local obstruction coming from the real numbers: the nature of our counting operator means that we cannot disentangle possible correlations between the $f$ and $g$ functions on subintervals of length $N^{1/2}$.  We can, however, show that these are the only other possible obstructions to uniformity.

\begin{theorem}[Local control of all terms]\label{local control}
Let $\delta > 0$ and $f_1, f_2, f_3 : \Z \to \C$ be 1-bounded functions with support in $\set{1, \dots, N}$.  Suppose that  
\[
\abs{\sum_{x\in\Z}\sum_{ y\in \N} f_1(x) f_2(x+y)f_3(x+y^2) }\geq \delta N^{3/2}.
\]
Then either $N \ll \delta^{-O(1)}$, or for each $i=1,2,3$ there is a 1-bounded local function $\phi_i$ of resolution $\gg \delta^{O(1)} N^{1/2}$ and modulus $q_i \ll \delta^{-O(1)}$ such that
$$
\abs{\sum_{x\in \Z} f_i(x)\phi_i(x) } \gg \delta^{O(1)} N.
$$
\end{theorem}

\begin{proof}
This is an immediate consequence of Corollary \ref{global control cor} and Lemma \ref{partial inverse theorem}.
\end{proof}

\subsection{Longer polynomial progressions}\label{longer intro}
In analogy with the first author's generalisation \cite{PeluseBounds} of \cite{PelusePrendivilleQuantitative}, it is natural to ask whether the methods of this paper yield polylogarithmic bounds for sets of integers lacking longer progressions
\begin{equation}\label{longer}
x,\ x+ P_1(y), \ \dots, \ x + P_m(y),% \qquad (y \in \Z\setminus{0}),
\end{equation} 
where the $P_i \in \Z[y]$ have zero constant term and $\deg P_1 < \dots < \deg P_m$.

As was mentioned above, the key input to this paper is the local $U^1$-control result \cite[Theorem 7.1]{PelusePrendivilleQuantitative}.  Replacing this with \cite[Theorem 3.3]{PeluseBounds}, our argument generalises in a straightforward manner to yield polylogarithmic bounds for subsets of $\{1,\dots,N\}$ lacking \eqref{longer} when $m = 2$, that is, for all three-term polynomial progressions with distinct degrees and zero constant term.

Obtaining polylogarithmic bounds for longer polynomial progressions requires an additional idea. We sketch a strategy in \S\ref{longer section}, which relies on obtaining an appropriate generalisation of \cite[Theorem 3.3]{PeluseBounds}, a generalisation that would require re-running the majority of the arguments therein.

\subsection*{Acknowledgements} S.~Peluse is supported by the NSF Mathematical Sciences Postdoctoral Research Fellowship Program under Grant No.\ DMS-1903038

\subsection{An outline of our argument}\label{sec3}

Effective Szemer\'edi-type theorems are commonly proved via a density increment strategy, the prototypical example being the proof of Roth's theorem \cite{RothCertainI} on three-term arithmetic progressions. This strategy begins with a set $A \subset \{1, \dots, N\}$ of density $\delta := |A|/N$ that lacks the configuration in question. It then proceeds to show that there is a substructure $S \subset \{1, \dots, N\}$ on which $A$ has increased density $\delta + \Omega_\delta(1)$. One then hopes to iterate the argument with $A \cap S$ in place of $A$ and $S$ in place of $\{1, \dots, N\}$. 

One avenue to obtaining polylogarithmic bounds in a Szemer\'edi-type theorem is to obtain  
a constant proportion density increment $\delta + \Omega(\delta)$ on a substructure $S$ of polynomial size $|S| \approx N^{\Omega(1)}$.  This was accomplished for three-term arithmetic progressions by Heath--Brown \cite{HeathBrownInteger} and Szemer\'edi  \cite{SzemerediInteger} (in fact, they were able to handle a smaller lower bound on $|S|$).

An alternative strategy for obtaining polylogarithmic bounds is to obtain the weaker polynomial increment $\delta + \Omega(\delta^{O(1)})$, yet on a \emph{dense} or \emph{global} substructure $S$, that is, a substructure of size $|S| \geq \exp(-O(\delta^{-O(1)})) N$. This was accomplished by S\'ark\"ozy \cite{SarkozyDifferenceI} for the configuration $x, x+y^2$ and for three-term arithmetic progressions by Bourgain \cite{BourgainTriples}. 

Both of these strategies are achievable for the nonlinear Roth configuration. The global structure strategy is perhaps the most natural, and may be accomplished by utilising a generalisation of Theorem \ref{global control}. In this note we do not pursue this, and instead give details for a constant-proportion density increment, as our argument is somewhat cleaner in this form.

More specifically,  we show that if $A \subset \set{1,\dots, N}$ has density $\delta$ and lacks nontrivial configurations of the form $x, x+y, x+y^2$, then there exists an arithmetic progression $P$ of length $|P| \gg \delta^{O(1)} N^{1/2}$ and common difference $q \ll \delta^{-O(1)}$ such that we have the density increment
\begin{equation}\label{constant proportion}
\frac{|A\cap P|}{|P|} \geq (1+\Omega(1)) \frac{|A|}{N}.
\end{equation}
As outlined in \cite{PelusePrendivilleQuantitative}, the `almost bounded' size of $q$ allows us to iterate this procedure. (In \cite{PelusePrendivilleQuantitative}, we obtain the weaker density increment $(1+\Omega(\delta^{O(1)}))|A|/N$, which leads to the extra logarithm appearing in the bound there.)

We obtain the constant-proportion increment \eqref{constant proportion} by combining the local $U^1$-control result of~\cite{PelusePrendivilleQuantitative} with a strategy of Heath--Brown \cite{HeathBrownInteger} and Szemer\'edi \cite{SzemerediInteger}, which has a very robust formulation due to Green and Tao \cite{GreenTaoNewII}. To accomplish this, we first give a structural characterisation of sets lacking the nonlinear Roth configuration (this is Lemma \ref{inverse theorem}, whose essence is captured in the weaker Theorem \ref{local control}). These sets resemble the level sets of the product of a function that is constant on intervals of length $N^{1/2}$ and a function that is constant on congruence classes modulo a bounded $q$.

Having obtained such a structural characterisation, an energy increment procedure closely following \cite{GreenTaoNewII} allows us to approximate an arbitrary set of integers by these level sets, up to an error that does not contribute substantially to the count of nonlinear Roth configurations.  A combinatorial argument then allows us to deduce that our set must have a substantial density increment on one of these level sets, of the form $\delta + \Omega(\delta)$.  As a result, our density increment procedure requires only $\log(\delta^{-1})+O(1)$ iterations, compared with the $O(\delta^{-O(1)})$ required in \cite{PelusePrendivilleQuantitative}, and this yields the polylogarithmic improvement over our previous density increment iteration.

The remainder of this paper is organized as follows. We derive Theorem \ref{main} in \S\ref{increment sec} via a density increment iteration. Our deduction uses a density increment lemma that is established in \S\S\ref{inverse theorem section}--\ref{increment proof sec}. We prove Theorem \ref{global control} and Corollary \ref{global control cor} in \S\ref{global sec}.

\subsection{Notation}\label{notation}
\subsubsection{Standard conventions}
We use $\N$ to denote the positive integers.  For a real number $X \geq 1$, write $[X] = \{ 1,2, \ldots, \floor{X}\}$.  A complex-valued function is said to be \emph{1-bounded} if the modulus of the function does not exceed 1. 

We use counting measure on $\Z$, so that for $f,g :\Z \to \C$, we have
$$
\norm{f}_{\ell^p} := \biggbrac{\sum_x |f(x)|^p}^{\recip{p}}, \ \ang{f,g} := \sum_x f(x)\overline{g(x)},\ \text{and}\ (f*g)(x) = \sum_y f(y)g(x-y).
$$ 
Any sum of the form $\sum_x$ is to be interpreted as a sum over $\Z$. The \emph{support} of $f$ is the set $\supp(f) := \set{x \in \Z : f(x) \neq 0}$. We write $\norm{f}_\infty$ for $\sup_{x \in \Z} |f(x)|$.
 
We use Haar probability measure on $\T := \R/\Z$, so that for measurable $F : \T \to \C$, we have
$$
\norm{F}_{L^p} := \biggbrac{\int_\T |F(\alpha)|^pd\alpha}^{\recip{p}} = \biggbrac{\int_0^1 |F(\alpha)|^pd\alpha}^{\recip{p}}.
$$
We write $\norm{\alpha}_\T$ for the distance from $\alpha \in \R$ to the nearest integer
$
\min_{n \in \Z} |\alpha - n|.
$
This remains well-defined on $\T$.

We define the Fourier transform of $f : \Z \to \C$ by 
\begin{equation}\label{Fourier transform}
\hat{f}(\alpha) := \sum_x f(x) e(\alpha x) \qquad (\alpha \in \T),
\end{equation}
when this makes sense.  Here $e(\alpha)$ stands for $e^{2\pi i \alpha}$.

For a finite set $S$ and function $f:S\to\C$, denote the average of $f$ over $S$ by
\[
\E_{s\in S}f(s):=\frac{1}{|S|}\sum_{s\in S}f(s).
\]

%
%If $\|\cdot\|$ is a seminorm on an inner product space, recall that its dual seminorm $\|\cdot\|^*$ is defined by
%\begin{equation}\label{dual norm}
%\|f\|^{*}:=\sup_{\|g\|\leq1}|\langle f,g\rangle|.
%\end{equation}
%Hence,
%\begin{equation}\label{dual ineq}
%\abs{\ang{f,g}} \leq \norm{f}^* \norm{g}.
%\end{equation}
%

For a complex-valued function $f$ and positive-valued function $g$, write $f \ll g$ or $f = O(g)$ if there exists a constant $C$ such that $|f(x)| \le C g(x)$ for all $x$. We write $f = \Omega(g)$ if $f \gg g$. We subscript this notation when the implicit constant may depend on the subscripted parameters. %We sometimes opt for a more explicit approach, using $C$ to denote a large absolute constant, and $c$ to denote a small positive absolute constant.  The values of $C$ and $c$ may change from line to line. 

\subsubsection{Local conventions}

Up to normalisation, all of the above are widely used in the literature. Next, we list notation specific to our paper. We have tried to minimise this in order to aid the casual reader.  

The quantity $(N/q)^{1/2}$ appears repeatedly, where $N$ and $q$ are integers fixed throughout the majority of our paper. We therefore adopt the convention that
\begin{equation}\label{M def}
M:= \floor{\sqrt{N/q}}.
\end{equation}
Assuming this, define the \emph{counting operator} on the functions $f_i : \Z \to \C$  by
\begin{equation}\label{counting op}
\Lambda_{q, N}(f_0, f_1, f_2) := \E_{x \in [N]} \E_{y \in [M]} f_0(x)f_1(x+y) f_2(x+qy^2).
\end{equation}
When $f_0=f_1=f_2=f$, we simply write $\Lambda_{q, N}(f)$ for $\Lambda_{q,N}(f_0,f_1,f_2)$.

For a real parameter $H \geq 1$, we use $\mu_H : \Z \to [0,1]$ to represent the following normalised Fej\'er kernel
\begin{equation}\label{fejer}
\mu_H(h) := \recip{\floor{H}} \brac{1 - \frac{|h|}{\floor{H}}}_+ = \frac{(1_{[H]} * 1_{-[H]} )(h)}{\floor{H}^2}.
\end{equation}
%For a multidimensional vector $h \in \Z^d$ we write
%\begin{equation}\label{multidim fejer}
%\mu_H(h) := \mu_H(h_1)\dotsm \mu_H(h_d).
%\end{equation}
This is a probability measure on $\Z$ with support in the interval $(-H, H)$.

\section{Iterating the density increment}\label{increment sec}
In this section we prove Theorem \ref{main} using the following lemma, which we will devote \S\S\ref{inverse theorem section}--\ref{increment proof sec} to proving.
\begin{lemma}[Density increment lemma]\label{increment lemma}
Let $q \leq N$ be positive integers and $\delta > 0$.  Suppose that $A \subset [N]$ satisfies $|A| \geq \delta N$ and lacks the configuration 
\begin{equation}\label{q config}
x,\ x+y, \ x+qy^2 \qquad (y \neq 0).
\end{equation} 
Then either  $N \ll (q/\delta)^{O(1)}$ or there exists $q' \leq \exp\brac{O\brac{\delta^{-O(1)}}}$ and $N' \geq q^{-O(1)}\exp\brac{-O\brac{\delta^{-O(1)}}}N^{1/2} $
 such that, for some $a \in \Z$, we have
\begin{equation}\label{increment}
 |A \cap (a + qq'\cdot[N'])| \geq (1+\Omega(1))\delta N'.
\end{equation}
\end{lemma}

\begin{proof}[Proof of Theorem \ref{main} given Lemma \ref{increment lemma}]
This is the same as the proof of~\cite[Theorem 1.1]{PelusePrendivilleQuantitative}, but using the improved density increment lemma above in place of the density increment lemma of~\cite{PelusePrendivilleQuantitative}. Note first that if $A$ lacks the configuration \eqref{q config}, then the set
\[
\{x :a+qq'x\in A\},
\]
lacks configurations of the form
\[
x,\ x+y,\ x+q^2q'y^2 \qquad (y \neq 0).
\]

Let $A \subset [N]$ have size $\delta N$, and suppose that it has no non-linear Roth configurations \eqref{main config}.  Setting $A_0 := A$, $N_0 := N$ and $q_0 = 1$, let us suppose we have a  sequence of tuples $(A_i, N_i, q_i)$ for $i = 0, 1, \dots, n$ that each satisfy the following:
\begin{enumerate}[(i)]
\item  $A_i$ lacks configurations of the form 
$$
x,\ x+y,\ x+q_0^{2^i} q_1^{2^{i-1}}\dotsm q_{i-1}^2 q_i y^2 \qquad (y \neq 0).
$$
\item\label{qi upper bound}  $q_i \leq \exp\brac{O\brac{\delta^{-O(1)}}}$;
\item  $A_i \subset [N_i]$ and for $i \geq 1$ we have 
$$
\frac{|A_i|}{N_i} \geq (1+c)\frac{|A_{i-1}|}{N_{i-1}},
$$
where $c = \Omega(1)$ is a positive absolute constant;
\item\label{length lower bound}  for $i \geq 1$ we have the lower bound 
$$
N_i \geq \frac{N_{i-1}^{1/2}}{\brac{q_0^{2^{i-1}}\dotsm q_{i-1}\exp\brac{\delta^{-O(1)}}}^{O(1)}}.
$$
\end{enumerate}  

Applying Lemma \ref{increment lemma} with $q = q_0^{2^i} q_1^{2^{i-1}}\dotsm q_{i-1}^2 q_i$, either 
\begin{equation}\label{termination condition}
N_n \ll \brac{q_0^{2^n}q_1^{2^{n-1}}\dotsm q_{n-1}^2 q_n/\delta}^{O(1)}, 
\end{equation}
or we may obtain $(A_{n+1}, N_{n+1}, q_{n+1})$ satisfying conditions (i)--(iv). If \eqref{termination condition} holds, then our iterative process terminates at stage $n$.  

If the number of iterations $n$ is at least $c^{-1}$, then the density of $A_n$ on $[N_n]$ is at least $2\delta$.  After an additional $\trecip{2}c^{-1}$ iterations, the density is at least $4\delta$.  Hence if the number of iterations is at least
$$
\ceil{c^{-1}} + \ceil{\trecip{2}c^{-1}} + \ceil{\trecip{4}c^{-1}}+ \dots + \ceil{\trecip{2^{m-1}}c^{-1}},
$$ 
then the density is at least $2^m\delta$.  The density therefore exceeds one if the number of iterations exceeds $2c^{-1} + \log_2(\delta^{-1})$.  Since this cannot happen, it follows that there exists $n \leq \log_2(\delta^{-1})+O(1)$ such that the procedure terminates at stage $n$.

At the point of termination, the  smallness assumption \eqref{termination condition} must hold, so that
\begin{equation*}\label{N_i upper bound}
N_n \leq  \exp\brac{O\Bigbrac{\delta^{-O(1)}}}.
\end{equation*}
On the other hand, iteratively applying the lower bound \eqref{length lower bound}, we have 
\begin{equation*}\label{N_i lower bound}
\begin{split}
N_n  & \geq \frac{N_{n-1}^{1/2}}{\brac{q_0^{2^{n-1}}\dotsm q_{n-1}\exp\brac{\delta^{-O(1)}}}^{O(1)}}\\
&  \geq N^{1/2^n}\sqbrac{q_0^{2^{n-1}}\dotsm q_{n-1}\exp\brac{\delta^{-O(1)}}}^{-O(1 + \recip{2} + \recip{4} + \dots + 2^{1-n}) } \\
& \gg \exp\brac{-O\brac{\delta^{-O(1)}}} N^{\Omega(\delta)},
\end{split}
\end{equation*}
where we use the upper bound \eqref{qi upper bound} on the $q_i$'s, together with $n \leq \log_2(\delta^{-1})+O(1)$. Taking a logarithm and comparing upper and lower bounds for $N_n$ gives
$
\log N \ll \delta^{-O(1)},
$
which yields the bound claimed in Theorem \ref{main}.
\end{proof}

\section{The cut norm inverse theorem}\label{inverse theorem section}

The first step of the proof of Lemma~\ref{increment lemma} is to use the main technical result of~\cite{PelusePrendivilleQuantitative} to prove an inverse theorem for the cut norm associated to $\Lambda_{q,N}$, which we now define.

\begin{definition}[Cut norm]
For positive integers $q \leq N$, we define the \emph{cut norm} of  $f : \Z \to \C$ by
\begin{equation}\label{q norm eq}
\norm{f}_{q, N} := \sup\{|\Lambda_{q, N}(f, g_1,g_2)|,\ |\Lambda_{q, N}(g_1, f,g_2)|,\ |\Lambda_{q, N}(g_1,g_2, f)|\},
\end{equation}
where the supremum is taken over all 1-bounded functions $g_i : [N] \to \C$.  We note that, in spite of our nomenclature, this is not a norm, but a seminorm.  One could remedy this by summing over $y \geq 0$ in the counting operator \eqref{counting op}.

Initially, the cut norm is too restrictive for us, so we begin by working with the weaker quantity
%\begin{equation}\label{sharp q norm eq}
%\norm{f}^\sharp_{q,N} := \sup\{ |\Lambda_{q, N}(g_0,g_1, f)|:  |g_i| \leq 1\ \text{ and }\ \supp(g_i) \subset [N] \}.
%\end{equation}
%and
\begin{equation}\label{flat q norm eq}
\norm{f}^\flat_{q,N} := \sup\{|\Lambda_{q, N}(f, g_1,g_2)|, |\Lambda_{q, N}(g_1, f,g_2)| :  |g_i| \leq 1 \text{ and } \supp(g_i) \subset [N] \},
\end{equation}
which we refer to as the \emph{partial cut norm}.
\end{definition}

The following lemma is simply a rephrasing of \cite[Theorem 7.1]{PelusePrendivilleQuantitative}, which is the technical heart of that paper. See Definition \ref{local function def} for the meaning of `local function'.

\begin{lemma}[Partial cut norm inverse theorem]\label{partial inverse theorem}
Let $q \leq N$ be positive integers, $\delta>0$, and $f:\Z\to\C$ be a $1$-bounded function with support in $[N]$.  
Suppose that  
\[
\norm{f}^\flat_{q,N}\geq\delta .
\]
Then either $N \ll (q/\delta)^{O(1)}$ or there exists a 1-bounded local function $\phi$ of resolution  $\gg (\delta/q)^{O(1)}N^{1/2}$, modulus $qq'$ for some $q'\ll  \delta^{-O(1)}$, and  such that
$$
\sum_{x\in[N]} f(x)\phi(x) \gg \delta^{O(1)} N.
$$
\end{lemma}

\begin{proof}
By compactness, there exist 1-bounded functions $g_1, g_2 : [N] \to \C$ 
such that either
$
|\Lambda_{q, N}(f, g_1, g_2)| \geq \delta
$ 
or
$
|\Lambda_{q, N}(g_1,f, g_2)| \geq \delta.
$ 
In the latter case, we may apply \cite[Theorem 7.1]{PelusePrendivilleQuantitative} to deduce that there exist positive integers $q'\ll\delta^{-O(1)}$ and $N' \gg (\delta/q)^{O(1)}N^{1/2}$ such that
\[
\sum_{x}\abs{\sum_{y\in[N']}f(x+qq'y)}\gg \delta^{O(1)}NN'.
\]
In the former case, the reader may check that the argument of \cite[Theorem 7.1]{PelusePrendivilleQuantitative} delivers the same conclusion\footnote{For details see the second author's exposition \cite{PrendivilleInverse}.}.

To ease notation, write $Q := qq'$. Partitioning the integers into arithmetic progressions of length $N'$ and common difference $Q$ gives
\begin{multline*}
\delta^{O(1)} NN'  \ll \sum_{z \in [N']}\sum_{u\in [Q]} \sum_{x\in\Z} \abs{\sum_{y \in [N']} f( Qz +QN'x +u + Qy)}\\  \leq N'\max_z\sum_{u\in [Q]} \sum_{x\in \Z} \abs{\sum_{y \in [N']} f(Qz + QN'x +u  +Qy)}.
\end{multline*}
Defining $ \psi_z(u,x)$ to be the conjugate phase of the inner sum, we deduce the existence of $z$ for which
\begin{align*}
\delta^{O(1)} N  \ll \sum_{u \in [Q]}\sum_x\sum_{y \in [N']}f(Qz + QN'x +u  +Qy)\psi_z(u,x).
\end{align*}
The result follows on noting that every integer has a unique representation of the form $QN'x +u  +Qy$ with $u \in [Q]$, $x \in \Z$ and $y \in [N']$.  Hence the map $$Qz + QN'x +u  +Qy \mapsto \psi_z(u,x)$$ is a local function of resolution $QN'$ and modulus $Q$.
\end{proof}

Now we can prove an inverse theorem for the cut norm itself.

\begin{lemma}[Full cut norm inverse theorem]\label{inverse theorem}
Let $q \leq N$ be positive integers, $\delta>0$, and $f:\Z\to\C$ be a $1$-bounded function with support in $[N]$.  
Suppose that  
\[
\norm{f}_{q,N}\geq\delta .
\]
Then either $N \ll (q/\delta)^{O(1)}$ or there exist 1-bounded local functions $\phi_1$ and $\phi_2$,  of resolution  $\gg (\delta/q)^{O(1)}N^{1/2}$ and moduli $qq_1$ and $qq_2$, respectively, for some $q_1,q_2\ll  \delta^{-O(1)}$ such that 
\begin{equation}\label{cut norm density increment}
\abs{\sum_{x\in[N]} f(x)\phi_1(x)\phi_2(x)} \gg \delta^{O(1)} N.
\end{equation}
\end{lemma}

\begin{proof}%[Proof of Lemma \ref{inverse theorem}]
By the definition of the cut norm \eqref{q norm eq} and Lemma \ref{partial inverse theorem}, we may assume that there are 1-bounded functions $g, h : [N] \to \C$ such that
\begin{equation}\label{pre-CS}
|\Lambda_{q, N}(g, h, f)|  \geq\delta .
\end{equation}
Recalling that $M := \lfloor\sqrt{N/q}\rfloor$, define the dual function
$$
F(x) : = \E_{y \in [M]} h(x+y) f(x+qy^2).
$$
Re-parametrising \eqref{pre-CS} and applying the Cauchy--Schwarz inequality, we have that
\begin{equation*}
\delta^2\leq \E_{x \in [N]} F(x)^2 = \E_{x \in [N]}\E_{y \in [M]}  F(x)h(x+y) f(x+qy^2).
\end{equation*}
Recalling the definition of the partial cut norm \eqref{flat q norm eq}, we deduce that
$$
\norm{F}_{q, N}^\flat \geq \delta^2.
$$

Applying the partial cut norm inverse theorem (Lemma \ref{partial inverse theorem}), there exists a 1-bounded local function $\phi_1$ of resolution $\gg (\delta/q)^{O(1)}N^{1/2}$ and modulus $qq_1$ for some $q_1\ll  \delta^{-O(1)}$ such that 
\[
\abs{\sum_{x\in[N]} F(x)\phi_1(x)} \gg \delta^{O(1)} N.
\]
Thus
$$
|\Lambda_{q, N}(\phi_1, h, f)|  \gg \delta^{O(1)} .
$$

We now re-run our argument on $h$ instead of $f$, deducing the existence of a 1-bounded local function $\phi_2$ of resolution $\gg (\delta/q)^{O(1)}N^{1/2}$ and modulus $qq_2$ for some $q_2\ll  \delta^{-O(1)}$ such that
$$
|\Lambda_{q, N}(\phi_1, \phi_2, f)|  \gg \delta^{O(1)} .
$$
Expanding the counting operator and taking a maximum over $y\in [M]$ gives
\begin{align*}
\delta^{O(1)} NM & \ll \abs{\sum_{y \in [M]}\sum_x f(x)\phi_1(x-qy^2) \phi_2(x-qy^2 + y)}\\
&  \leq M \abs{\sum_x f(x)\tilde{\phi}_1(x) \tilde{\phi}_2(x)},
\end{align*}
where both $\tilde{\phi}_i$ are 1-bounded local functions of resolution $\gg (\delta/q)^{O(1)}N^{1/2}$ and moduli $qq_i$ for some $q_i\ll \delta^{-O(1)}$.
\end{proof}

\section{A weak regularity lemma}

Much of the material is this section is standard, and closely follows the expositions in Green \cite{GreenMontreal} and Green--Tao \cite{GreenTaoNewII}. To simplify the exposition of later arguments, while the factors in~\cite{GreenMontreal} and~\cite{GreenTaoNewII} are $\sigma$-algebras, our factors will be the set of atoms of certain $\sigma$-algebras (which can obviously be recovered by taking the $\sigma$-algebra generated by the set of atoms).

\begin{definition}[Factor]
  We define a \emph{factor} $\mathcal{B}$ of $[N]$ to be a partition of $[N]$, so that $[N] = \sqcup_{B \in \mathcal{B}} B$.   We say that a factor $\mathcal{B}'$ \emph{refines} $\mathcal{B}$ if every element of $\mathcal{B}$ is a union of elements of $\mathcal{B}'$.  The \emph{join} $\mathcal{B}_1\vee\dots\vee\mathcal{B}_d$ of factors $\mathcal{B}_1, \dots, \mathcal{B}_d$ is the factor formed by taking the $d$-fold intersections of the elements of $\mathcal{B}_1$, \dots, $\mathcal{B}_d$, that is,
  \[
\mathcal{B}_1\vee\dots\vee\mathcal{B}_d:=\{B_1\cap\dots\cap B_d:B_i\in\mathcal{B}_i\text{ for }i=1,\dots,d\}.
  \]
  
\end{definition}

\begin{definition}[Measurability, projection]
Given a factor $\mathcal{B}$, we say that a function $f : [N] \to \C$ is \emph{$\mathcal{B}$-measurable} if it is constant on the elements of $\mathcal{B}$.  

Define the \emph{projection} of any function $f : [N] \to \C$ onto $\mathcal{B}$  by
\begin{equation}\label{conditional expectation}
\Pi_{\mathcal{B}} f(x) = \E_{y \in B_x} f(y),
\end{equation}
where $B_x$ is the element of $\mathcal{B}$ that contains $x$.
Notice that $\Pi_{\mathcal{B}} f$ is $\mathcal{B}$-measurable, and is just the conditional expectation of $f$ with respect to the $\sigma$-algebra generated by the elements of $\mathcal{B}$.
\end{definition}

We record some well-known properties of the projection operator $\Pi_{\mathcal{B}}$ (that is, properties of conditional expectation) in the next lemma.

\begin{lemma}[Properties of the projection operator]\label{properties} {\ }
\begin{enumerate}[{\normalfont (i)}]
\item\label{projection part} The operator $\Pi_{\mathcal{B}}$ linearly projects onto the space of $\mathcal{B}$-measurable functions.
\item\label{adjoint part}  $\Pi_{\mathcal{B}}$ is self-adjoint with respect to the inner product
$$
\ang{f, g} := \sum_x f(x) \overline{g(x)} \qquad (f, g : [N] \to \C),
$$
so that $\ang{f, \Pi_{\mathcal{B}}g} = \ang{\Pi_{\mathcal{B}} f, g}$.
\item  If $\mathcal{B}'$ is a refinement of $\mathcal{B}$ then 
$$
\Pi_{\mathcal{B}'}\Pi_{\mathcal{B}}f  = \Pi_{\mathcal{B}}f. %= \Pi_{\mathcal{B}}\Pi_{\mathcal{B}'}f .
$$
\item\label{orthog part}  If $\mathcal{B}'$ refines $\mathcal{B}$ then   $\Pi_{\mathcal{B}}f $ is orthogonal to $\Pi_{\mathcal{B}'}f  - \Pi_{\mathcal{B}}f $.
\end{enumerate}
\end{lemma}

\begin{proof}
Inspecting the formula \eqref{conditional expectation} reveals that $\Pi_{\mathcal{B}}$ is linear, that $\Pi_{\mathcal{B}}f$ is constant on elements of $\mathcal{B}$, and that if $f$ itself is constant on elements of $\mathcal{B}$, then $\Pi_{\mathcal{B}}f = f$.  This establishes \eqref{projection part}.

Interchanging the order of summation gives
\begin{equation*}
\begin{split}
\ang{f, \Pi_{\mathcal{B}}g} = \sum_{B \in \mathcal{B}} |B|^{-1} \sum_{x,y \in B} f(x)\overline{g(y)} = \ang{\Pi_{\mathcal{B}}f, g }.
\end{split}
\end{equation*}
This proves that $\Pi_\mathcal{B}$ is self-adjoint.

The first refinement property follows from the fact that $\Pi_{\mathcal{B}} f$ is $\mathcal{B}'$-measurable.

We utilise self-adjointness of $\Pi_{\mathcal{B}}$ and the first refinement property to conclude that
\begin{equation*}
\begin{split}
\ang{\Pi_{\mathcal{B}} f  , \Pi_{\mathcal{B}} f- \Pi_{\mathcal{B}'} f}& = \ang{\Pi_{\mathcal{B}} f  , \Pi_{\mathcal{B}} f-  f} = \ang{f  , \Pi_{\mathcal{B}} f-  \Pi_{\mathcal{B}} f} = 0.
\end{split}
\end{equation*}

\end{proof}

Now we describe the particular type of factors that will be relevant to us.

\begin{definition}[Local factor]\label{local factor def}
A  \emph{simple real factor} of resolution $M$ is a factor of $[N]$  obtained by partitioning  $\R$ into intervals all of length $M$.

A  \emph{simple congruence factor} of modulus $q$ is the factor of $[N]$ obtained by partitioning into congruence classes mod $q$.

We say that $\mathcal{B}$ is a \emph{simple local factor} of resolution $M$ and modulus $q$ if it is the join of a simple real factor of resolution $M$ and a simple congruence  factor of modulus $q$.  Notice that $\mathcal{B}$ is a simple local factor if and only if it consists of the level sets of a local function (Definition \ref{local function def}) of resolution $M$ and modulus $q$.

A \emph{local factor} of  dimension $d$, resolution $M$ and modulus $q$ is the join of $d$ simple local factors $\mathcal{B}_i$, each of resolution $M_i$ and modulus $q_i$, where $M_i \geq M$ and $q = \lcm[q_1, \dots, q_d]$. 
\end{definition}

Local factors of large resolution and small modulus and dimension necessarily contain few sets. This fact will be useful later in the proof of Lemma~\ref{increment lemma}.

\begin{lemma}[Size of a local factor]
\label{atom bound}
If $\mathcal{B}$ is a local factor of dimension $d$, resolution $M$, and modulus $q$, then
$$
|\mathcal{B}| \leq qd\brac{\frac{N}{M} + 2}.
$$ 
\end{lemma}

\begin{proof}
By the definition of a local factor, %(Definition \ref{local factor def}), 
it suffices to bound the size of the join of $d$ simple real factors, and then bound the size of the join of $d$ simple congruence factors.  The product of these two numbers gives us our final bound.

Joining $d$ congruence simple factors with moduli $q_1, \dots, q_d$ results in another congruence simple factor of modulus $q= \lcm[q_1, \dots, q_d]$.  The number of parts in such a partition is $q$.

The join of $d$ simple real factors partitions $[N]$ into intervals.  The upper endpoint of each of these intervals is either equal to $N$ or is equal to an endpoint of an interval in one of the original simple real factors.  For a simple real factor of resolution $M$, at most $1+N/M$ upper endpoints lie in $[1, N)$.  Hence the number of intervals in the join of $d$ simple real factors of resolutions $M_1$, \dots, $M_d$ is at most $2d + N(M_1^{-1} + \dots + M_d^{-1})$.\end{proof}

We now prove a weak regularity lemma for the cut norm via an energy increment argument.

\begin{lemma}[Weak regularity]\label{weak regularity}
Let $q \leq N$ be positive integers and $\delta >0$. Either $N \ll (q/\delta)^{O(1)}$, or for any function $f : [N] \to [0, 1]$ there exists a local factor $\mathcal{B}$ of dimension $d \ll \delta^{-O(1)}$, resolution $\gg (\delta/q)^{O(1)}N^{1/2}$, and modulus $qq'$ for some $q'\leq O\brac{1/ \delta}^{O(d)}$ such that
\begin{equation}\label{proj error}
\norm{f - \Pi_\mathcal{B} f}_{q, N} \leq \delta.
\end{equation}
\end{lemma}

\begin{proof}
We run an energy increment argument, initialising at stage $0$ with the trivial factor $\mathcal{B}_0 := \set{[N]}$.  Suppose that at stage $d$ of this iteration we have a local factor $\mathcal{B}$ of resolution  $ \gg (\delta/q)^{O(1)}N^{1/2}$, dimension at most $2d$, and modulus $qq'$ for some $q'\leq O(1/  \delta)^{O(d)}$.  In addition, suppose that we have the energy lower bound
\begin{equation}\label{energy bound}
\norm{\Pi_{\mathcal{B}} f}_{\ell^2}^2 \gg d\delta^{O(1)} N.
\end{equation}

With these assumptions in place, we query if the following holds
\begin{equation}\label{query}
\norm{f - \Pi_{\mathcal{B}}f}_{q, N} \leq \delta .
\end{equation} 
If so, then the process terminates.  If not, we show how our iteration may proceed to stage $d+1$.  

Applying the cut norm inverse theorem (Lemma \ref{inverse theorem}), we conclude that there exist 1-bounded local functions $\phi_i$ of  resolution  $ \gg (\delta/q)^{O(1)}N^{1/2}$ and modulus $qq_i$ for some $q_i\leq  \delta^{-O(1)}$ such that
$$
\abs{\ang{f - \Pi_{\mathcal{B}}f, \phi_1\phi_2}} = \abs{\sum_{x \in [N]} (f - \Pi_{\mathcal{B}}f)(x) \phi_1(x)\phi_2(x)} \gg \delta^{O(1)} N.
$$

Let $\mathcal{B}' $ denote the join of $\mathcal{B}$ and the simple local factors generated by $\phi_1$ and $\phi_2$, so that $\mathcal{B}'$ is a local factor of dimension at most $2(d+1)$, resolution $ \gg (\delta/q)^{O(1)}N^{1/2}$ and modulus $qq''$ for some $q'' \leq q'q_1q_2\leq  O(1/\delta)^{O(d+1)}$. Since $\phi_1\phi_2$ is $\mathcal{B}'$-measurable, we can use the properties listed in Lemma \ref{properties} together with the Cauchy--Schwarz inequality to deduce that
\begin{equation*}
\begin{split}
\abs{\ang{f- \Pi_{\mathcal{B}}f ,\phi_1\phi_2}} & =\abs{\ang{f- \Pi_{\mathcal{B}}f ,\Pi_{\mathcal{B}'}(\phi_1\phi_2)}} = \abs{\ang{\Pi_{\mathcal{B}'} f - \Pi_{\mathcal{B}} f , \phi_1\phi_2}}\\
& \leq N^{1/2}  \norm{\Pi_{\mathcal{B}'} f - \Pi_{\mathcal{B}}f }_{\ell^2}.
\end{split}
\end{equation*}
It follows that 
$$
\norm{\Pi_{\mathcal{B}'} f - \Pi_{\mathcal{B}}f}_{\ell^2} \gg\delta^{O(1)} N^{1/2}.
$$

Lemma \ref{properties} \eqref{orthog part} tells us that $\Pi_{\mathcal{B}} f$ is orthogonal to $\Pi_{\mathcal{B}'}f -  \Pi_{\mathcal{B}} f$, hence by  Pythagoras's theorem 
$$
\norm{\Pi_{\mathcal{B}'} f}_{\ell^2}^2 = \norm{\Pi_{\mathcal{B}}f }_{\ell^2}^2+ \norm{\Pi_{\mathcal{B}'} f- \Pi_{\mathcal{B}}f }_{\ell^2}^2.
$$
The energy bound \eqref{energy bound} follows for $\mathcal{B}'$, allowing us to proceed to the next stage of our iteration.

Since the function $f$ is $1$-bounded, the projection $\Pi_\mathcal{B}f$ is  also 1-bounded, hence the energy \eqref{energy bound} is always bounded above by $N$.  It follows that this energy increment must terminate at stage $d$ for some $d \ll \delta^{-O(1)}$, yielding the lemma.\end{proof}

\section{The density increment lemma}\label{increment proof sec}
In this section we prove Lemma \ref{increment lemma}, modelling our  argument on that given by Green and Tao \cite[Corollary 5.8]{GreenTaoNewII}. We first record, for the sake of convenience, the following immediate consequence of the triangle inequality.

\begin{lemma}[$\ell^1$-control]\label{L1 control}
Suppose that $N \geq q$.  Then for any $f_0, f_1, f_2 : [N]\to \C$ we have
$$
|\Lambda_{q, N}(f_0, f_1, f_2)| \leq N^{-1}\norm{f_i}_{\ell^1} \prod_{j \neq i} \norm{f_j}_\infty .
$$
\end{lemma}

\begin{proof}
We prove the result for $i = 1$, the other cases being similar.  A reparametrisation gives
\begin{align*}
\abs{\Lambda_{q, N}(f_0, f_1, f_2)} &= \abs{ \E_{x\in[N]} f_1(x) \E_{y \in [M]} f_0(x-y)f_2(x+qy^2 - y)}\\
 & \leq \E_{x\in[N]} |f_1(x)| \E_{y \in [M]} |f_0(x-y)||f_2(x+qy^2 - y)|.
\end{align*}
\end{proof}
We are now in a position to prove Lemma \ref{increment lemma}, and thereby complete our proof of Theorem \ref{main}.
\begin{proof}[Proof of Lemma \ref{increment lemma}]
Let $A$ satisfy the assumptions of Lemma \ref{increment lemma}.  Increasing $\delta$ only strengthens our conclusion, so we may assume that $|A| = \delta N$.  Since $\Lambda_{q, N}(1_A) = 0$, we have that
$
\abs{\Lambda_{q, N}(1_A)  -\Lambda_{q, N}(\delta1_{[N]})} =  \delta^3 \Lambda_{q, N}(1_{[N]}) \gg \delta^3
$.

Applying the weak regularity lemma (Lemma \ref{weak regularity}), there exists a local factor $\mathcal{B}$ of dimension $d \ll \delta^{-O(1)}$, resolution $\gg (\delta/q)^{O(1)}N^{1/2}$,  and  modulus $qq'$ for some $q' \leq O(1/ \delta)^{O(d)}$ such that
$$
\norm{1_A - \Pi_\mathcal{B} 1_A}_{q, N} \leq\tfrac{1}{6} \delta^3 \Lambda_{q, N}({1_{[N]}}).
$$
Setting $f := \Pi_\mathcal{B} 1_A$, a telescoping identity thus yields
$$
\abs{\Lambda_{q, N}(f)  -\Lambda_{q, N}(\delta1_{[N]})} \geq \trecip{2}  \delta^3 \Lambda_{q, N}({1_{[N]}}) \gg \delta^3.
$$

Define the $\mathcal{B}$-measurable set
$$
S:= \set{x \in [N] : f(x) \geq (1+c)\delta},
$$
where $c>0$ is a sufficiently small absolute constant that will be chosen to make the following argument valid.  By Lemma \ref{L1 control} and a telescoping identity, we have $\abs{\Lambda_{q, N}(f)  -\Lambda_{q, N}(f1_{S^c})} \leq 3|S|/N$, so that
$$
\tfrac{|S|}{N} + \abs{\Lambda_{q, N}(f1_{S^c})  -\Lambda_{q, N}(\delta1_{[N]})} \gg  \delta^3 .
$$
Yet another telescoping identity, in conjunction with  Lemma \ref{L1 control}, gives 
\begin{align*}
\abs{\Lambda_{q, N}(f1_{S^c})  -\Lambda_{q, N}(\delta 1_{[N]})} & \ll \tfrac{\delta^2 }{N}
\norm{f1_{S^c}- \delta1_{[N]}}_{\ell^1} \leq \tfrac{\delta^2 }{N} \norm{f - \delta1_{[N]}}_{\ell^1} + \tfrac{|S|}{N},
\end{align*}
so that
$$
 |S| + \delta^2 \norm{f - \delta1_{[N]}}_{\ell^1}  \gg \delta^3 N.
$$

Since $f - \delta1_{[N]}$ has mean zero, its $\ell^1$-norm is equal to twice the $\ell^1$-norm of its positive part. The function $\brac{f - \delta1_{[N]}}_+$ can only exceed $c \delta$ on $S$, so taking $c$ small enough gives 
$
|S| \gg \delta^3N
$. 
Letting $B$ denote the largest element of $\mathcal{B}$ for which $B \subset S$, the bound in Lemma \ref{atom bound}  yields
$$
|B| \gg q^{-O(1)}\delta^{O(d)} 2^{-O(d)} N^{1/2}.
$$
By construction (see Definition \ref{local factor def}), the set $B$ is an arithmetic progression of common difference $qq'$ with $q' \leq O(1/\delta)^{O(d)}$.  Moreover, the density of $A$ on $B$ is equal to the value of $f(x)$ for any $x \in  B$, and this is at least $(1+c)\delta$ by the definition of $S$.
\end{proof}

\section{Global control by major arc Fourier coefficients}\label{global sec}

The purpose of this section is to prove Theorem \ref{global control} and Corollary \ref{global control cor}. We begin with an alternative version of Lemma \ref{partial inverse theorem}, replacing the rigid local function found therein with something more continuous. 
\begin{definition}[$C$-Lipschitz]
We say that $\phi : \Z \to \C$ is \emph{$C$-Lipschitz along $q \cdot \Z$} if for any $x, y \in \Z$ we have
$$
|\phi(x+qy) - \phi(x)| \leq C |y|.
$$
\end{definition}
Recalling our definition for the Fej\'er kernel \eqref{fejer}, we observe that a function of the form 
\begin{equation}\label{h def}
x \mapsto \sum_h \mu_H(h) f(x+qh)
\end{equation} 
is Lipschitz along $q \cdot \Z$.

\begin{lemma}\label{h lipschitz}
Let $q, H$ be positive integers and $f : \Z \to \C$ be 1-bounded.
If $\phi$ is defined as in \eqref{h def}, then $\phi$ is $O(H^{-1})$-Lipschitz along $q \cdot \Z$.
\end{lemma}

\begin{proof} 
  Recalling \eqref{fejer}, the triangle inequality for $|\cdot|$ and $\max\{\cdot, 0\}$ show that  $|\mu_H(h+y) - \mu_H(h)| \leq |y|/\floor{H}^2$ for all $h, y \in \Z$.  Hence a change of variables gives
$$
|\phi(x+qy) - \phi(x)|  \leq  \sum_h |\mu_H(h-y) - \mu_H(h)| \ll \frac{|y|}{H^2}\sum_{h \in (-H, H)\cup(y-H, y+H) } 1.
$$
\end{proof}

Now we prove another partial cut norm inverse theorem, this time getting correlation with functions that are Lipschitz along progressions with small common difference.

\begin{lemma}[Partial cut norm inverse theorem II]\label{partial inverse theorem II}
Let $N$ be a positive integer, $ \delta>0$, and $f, g, h:\Z\to\C$ be $1$-bounded functions with support in $[N]$.  
Suppose that  
\[
\abs{\E_{x \in [N]} \E_{y \in [N^{1/2}]} f(x) g(x+y) h(x+y^2)}\geq\delta .
\]
Then either $N \ll \delta^{-O(1)}$, or there exists $q \ll \delta^{-O(1)}$ and  a 1-bounded function $\phi$ that is $O(\delta^{-O(1)} N^{-1/2})$-Lipschitz along $q \cdot \Z$ such that 
$$
\sum_{x\in[N]} g(x)\phi(x) \gg \delta^{O(1)} N.
$$
\end{lemma}

\begin{proof}
Applying \cite[Theorem 7.1]{PelusePrendivilleQuantitative}, we obtain positive integers $q\ll\delta^{-O(1)}$ and $N^{1/2} \geq M \gg \delta^{O(1)}N^{1/2}$ such that
\[
\sum_{x}\abs{\sum_{y\in[M]}g(x+qy)}\gg \delta^{O(1)}NM.
\]
By the Cauchy--Schwarz inequality and a change of variables, we have
$$
\sum_{x}g(x)\sum_{y_1, y_2\in[M]}\overline{g(x+q(y_1 - y_2))}\gg \delta^{O(1)}NM^2.
$$
Setting
$$
\phi(x) := \E_{y_1, y_2\in[M]}\overline{g(x+q(y_1 - y_2))},
$$
Lemma \ref{h lipschitz} shows this function has the required properties.
\end{proof}

Before proving Theorem~\ref{global control}, we record two standard facts.

\begin{lemma}\label{rog lem}
There are at most  $O(N^4)$ solutions $x \in [N]^6$ to the equation
$$
 x_1^2+x_2^2+x_3^2 = x_4^2 +x_5^2 + x_6^2.
$$
\end{lemma}

\begin{proof}
There are a number of ways to prove this. Perhaps the most robust is via the circle method, see \cite{DavenportAnalytic}. The result can be read out of \cite[Proposition 1.10]{BourgainLambda}.
\end{proof}

\begin{lemma}[Weyl's inequality]\label{weyl-ineq}
Let $P \subset \Z$ be an arithmetic progression with common difference $q$ and let $0< \delta \leq 1$.  Suppose that 
$$
\abs{\sum_{x \in P} e(\alpha x^2) } \geq \delta |P|.
$$
Then either $|P| \ll \delta^{-O(1)}$ %or $q \gg \delta^{-\Omega(1)}$ 
or there exists  a positive integer $q' \ll \delta^{-O(1)}$ such that  
$$\|q'q^2\alpha\| \ll \delta^{-O(1)}|P|^{-2}.$$
\end{lemma}

\begin{proof}
Let $P = x_0 + q\cdot [N]$, so that our exponential sum becomes
$$
\sum_{x \in P} e(\alpha x^2) = \sum_{y \in [N]} e(\alpha q^2 y^2 + 2\alpha q x_0 y + \alpha x_0^2).
$$
Applying \cite[Lemma A.11]{GreenTaoQuadratic}, either $N \ll \delta^{-O(1)}$ or the conclusion of our lemma follows.
\end{proof}

\begin{proof}[Proof of Theorem \ref{global control}]
Write $\Lambda_N$ for the counting operator $\Lambda_{1, N}$ (that is, the average \eqref{counting op} with $q=1$).  Let $f, g, h : [N] \to \C$ be 1-bounded functions satisfying
$$
|\Lambda_{N}(f, g, h) |  \geq\delta .
$$
Define the seminorm 
$$
\norm{g} := \sup\set{|\Lambda_N(g_1, g, g_2)|: |g_i| \leq 1 \text{ and } \supp(g_i) \subset [N]}.
$$
and the dual function
$$
F(x) : = \E_{y \in [N^{1/2}]} f(x-y) h(x+y^2-y).
$$
We follow the argument in the proof of Lemma \ref{inverse theorem} to  deduce that
$$
\norm{F} \geq \delta^2 .
$$
Hence, by Lemma \ref{partial inverse theorem II}, there exists $q \ll \delta^{-O(1)}$ and  a 1-bounded function $\phi$ that is $O(\delta^{-O(1)} N^{-1/2})$-Lipschitz along $q \cdot \Z$ and satisfies 
$$
\sum_{x\in[N]} F(x)\phi(x) \gg \delta^{O(1)} N.
$$
Expanding the definition of the dual function, we have
$$
 \sum_{x\in [N]}\sum_{y \in [N^{1/2}]}f(x)\phi(x+y) h(x+y^2) \gg \delta^{O(1)}N^{3/2}.
$$

Let us partition $\Z$ into arithmetic progressions $P$ each of common difference $q$ and length $M$, where $M$ will be chosen shortly. For each such arithmetic progression $P$, fix an element $y_P \in P$. Using the Lipschitz property of $\phi$, for any $x \in \Z$ and $y \in P$ we have
$$
 |\phi(x+y_P) - \phi(x+y)| \ll \delta^{-O(1)}M N^{-1/2}.
$$
Hence,
$$
 \abs{\sum_P\sum_{x\in [N]}\sum_{y \in P\cap [N^{1/2}]}f(x)[\phi(x+y) - \phi(x + y_P) ]h(x+y^2)} \ll \delta^{-O(1)}MN.
$$
We can therefore take $M$ sufficiently small to satisfy both $M \gg \delta^{O(1)} N^{1/2}$ and
$$
\abs{\sum_P\sum_x\sum_{y \in P\cap [N^{1/2}]}f(x) \phi(x + y_P) h(x+y^2)} \gg \delta^{O(1)}N^{3/2}.
$$

Set $f_P(x) := f(x) \phi(x+y_P)$.  The number of progressions $P$ that intersect $[N^{1/2}]$ is at most  $O(N^{1/2}M^{-1} + q) = O(\delta^{-O(1)})$. Therefore, the pigeon-hole principle gives a progression $P$ for which
\begin{equation}\label{shorter sarko}
\abs{\sum_x\sum_{y \in P\cap[N^{1/2}]} f_P(x) h(x+y^2)} \gg \delta^{O(1)}N^{3/2}.
\end{equation}
  In particular, $|P \cap [N^{1/2}]| \gg \delta^{O(1)} N^{1/2}$.

Writing $S_{P}(\alpha)$ for $\sum_{y \in P\cap [N^{1/2}]} e\brac{\alpha y^2}$, the orthogonality relations allow us to reformulate \eqref{shorter sarko} as
\begin{align*}
 \abs{\int_\T \hat{f}_P(\alpha) \hat{h}(-\alpha) S_{P}(\alpha) d \alpha} \gg \delta^{O(1)}N^{3/2}.
\end{align*}
Let $\eta > 0$ be a parameter to be determined shortly, and define the major arcs 
$$
\mathfrak{M} := \set{\alpha \in \T : |S_{P}(\alpha)| \geq \eta N^{1/2}}.
$$
 Parseval's identity then gives
$$
\abs{\int_{\T\setminus\mathfrak{M}} \hat{f}_P(\alpha) \hat{h}(-\alpha) S_{P}(\alpha) d \alpha} \leq \eta N^{1/2} \bignorm{\hat{f}_P}_2\bignorm{\hat{h}}_2 \leq \eta N^{3/2}.
$$
Hence we may take $\eta \gg \delta^{O(1)}$ and ensure that 
\begin{align*}
 \abs{\int_{\mathfrak{M}} \hat{f}_P(\alpha) \hat{h}(-\alpha) S_{P}(\alpha)d \alpha} \gg \delta^{O(1)}N^{3/2}.
\end{align*}

By Lemma \ref{rog lem} and orthogonality, we have $\norm{S_{P}}_6 \ll  N^{1/3}$. Thus, by H\"older's inequality, we get that
$$
\abs{\int_{\mathfrak{M}} \hat{f}_P(\alpha) \hat{h}(-\alpha) S_{P}(\alpha)d \alpha} \leq  \bignorm{\hat{f}_P}_2 \bignorm{\hat{h}}_2^{2/3} \bignorm{S_P}_6 \sup_{\alpha \in \mathfrak{M}}  \bigabs{\hat{h}(-\alpha)}^{1/3}.
$$
We therefore deduce that there exists $\alpha \in \mathfrak{M}$ such that
$$
\bigabs{\hat{h}(-\alpha)} \gg \delta^{O(1)} N.
$$

Finally, an application of Weyl's inequality (Lemma \ref{weyl-ineq}) shows that if $-\alpha \in \mathfrak{M}$ then $\alpha$ has the required Diophantine approximation property.
\end{proof}

\begin{proof}[Proof of Corollary \ref{global control cor}]
Let $\alpha\in\R$ be the frequency and $q$ the positive integer provided by Theorem \ref{global control}.  For any integer $a$ and positive integer $M$, if $x, y \in a + q\cdot[M]$, then
 $$
 \abs{e(\alpha x) - e(\alpha y)} \leq 2\pi\norm{\alpha(x-y)} \ll \delta^{-O(1)} M N^{-1}.
 $$
 Partitioning $\Z$ into arithmetic progressions of common difference $q$ and length $M$ then gives
 $$
 \delta^{O(1)} N \ll \sum_P\Bigabs{\sum_{x \in P} h(x)} +   \delta^{-O(1)} M .
 $$
We thus take $M \gg \delta^{O(1)} N$ sufficiently small to ensure that
$$
 \delta^{O(1)} N \ll \sum_P\Bigabs{\sum_{x \in P} h(x)}.
 $$
Write $\theta_P$ for the conjugate phase of the inner sum.  Then the map $x \mapsto \sum_P\theta_P1_P(x)$ is a local function of resolution $\gg \delta^{O(1)}N$ and modulus $\ll \delta^{-O(1)}$, yielding the corollary.
\end{proof}

\section{Longer progressions}\label{longer section}
As mentioned in \S\ref{longer intro}, the main obstacle to generalising our polylogarithmic bound to longer configurations such as \eqref{longer} is in obtaining an appropriate generalisation of Lemma \ref{inverse theorem}; in particular, showing that if the relevant counting operator is large, then \emph{all} functions must correlate with a product of a bounded number of local functions.

Let us demonstrate where the argument breaks down for $m > 2$. Given polynomials as in \eqref{longer} and 1-bounded functions $f_0, f_1, \dots, f_m : [N] \to \C$, define the counting operator
\begin{multline*}
\Lambda_{P_1, \dots, P_m}^N(f_0, f_1, \dots, f_m) := \\\E_{x \in [N]} \E_{y \in [N^{1/\deg P_m}]}f_0(x)f_1(x+P_1(y))\dotsm f_m(x+P_m(y)) .
\end{multline*}
 Using the main technical result of~\cite{PeluseBounds}, \cite[Theorem 3.3]{PeluseBounds}, one can show that if 
$$
\abs{\Lambda_{P_1, \dots, P_m}^N(f_0, f_1, \dots , f_m)} \geq \delta ,
$$
then both $f_0$ and $f_1$ correlate with local functions $\phi_0$ and $\phi_1$. Combining this with a dual function argument, as in our proofs of Theorem \ref{global control} and Lemma \ref{inverse theorem}, one may conclude that 
$$
\abs{\Lambda_{P_1, \dots, P_m}^N(\phi_0, \phi_1, f_2, \dots , f_m)} \gg \delta^{O(1)},
$$
If $m = 2$, one can then pigeon-hole in the smaller $y$ variable appearing in the counting operator (as we do in the proof of Lemma \ref{inverse theorem}) to conclude that $f_2$ correlates with a product of two local functions. It is this simple pigeon-holing argument that fails when $m > 2$.

\subsection{An alternative strategy for longer progressions}
A more productive strategy is to follow our proof of Theorem \ref{global control} instead of Theorem \ref{main}.  In proving Theorem \ref{global control} we replace the counting operator $\Lambda_{y, y^2}^N(f_0, f_1, f_2)$ with $\Lambda_{y, y^2}^N(f_0, \phi, f_2)$, where $\phi$ is a local function that is constant on progressions of length $\approx N^{1/2}$ with common difference of size $\approx O(1)$.  Provided that we pass to appropriate subprogressions in all of the variables appearing in our counting operator, we can exploit the properties of this local function and `remove' it from our count. In effect (after passing to subprogressions of bounded common difference), we replace the count $\Lambda_{y, y^2}^N(f_0, f_1, f_2)$ with one of the form $\Lambda_{Q}^{N'}(f_0, f_2)$, where $Q$ is a quadratic polynomial and $N'$ is slightly smaller than $N$. 

Generalising this approach, one can use \cite[Theorem 3.3]{PeluseBounds} to replace the counting operator $\Lambda_{P_1, \dots, P_m}^N(f_0, f_1, \dots , f_m)$ with $\Lambda_{P_1, \dots, P_m}^N(f_0, \phi, f_2, \dots , f_m)$, where $\phi$ is a local function. Provided that this local function has resolution $\gg N^{\deg P_1 / \deg P_m}$ and common difference $q \ll 1$, we have
$$
\phi(x+P_1(y)) \approx \phi(x) 
$$
for any $x\in \Z$ and any $y$ constrained to a subprogression of common difference $q$ and length $\approx N^{\deg P_1 / \deg P_m}$.
Passing to subprogressions in $x$ and $y$, one should then be able to replace the operator $$\Lambda_{P_1, \dots, P_m}^N(f_0, \phi, f_2, \dots , f_m)$$ by one of the form $$\Lambda_{Q_2, \dots, Q_m}^{N'}(f_0, f_2, \dots , f_m).$$
Applying induction on $m$ may then allow one to show that every function in the original counting operator correlates with a local function.

The main impediment to carrying out this strategy is that the polynomials $Q_2$, \dots, $Q_m$, which arise on passing to a subprogression, may not satisfy the hypotheses required to reapply \cite[Theorem 3.3]{PeluseBounds}. It is likely that the polynomials are sufficiently well-behaved for the arguments of \cite{PeluseBounds} to remain valid, but we leave this verification to the energetic reader.

%\appendix

%%\renewcommand{\refname}{\normalsize References}
%%\bibliographystyle{plain}
%%\bibliography{bib}
%{
%%\footnotesize%\small%\scriptsize
% %\bibliographystyle{habbrv}
%% \bibliographystyle{abbrvnat-noDOIorURL}
% %\bibliographystyle{abbrv}
%  %\bibliographystyle{acm}
% % \bibliographystyle{hsiam}  
% %  \bibliographystyle{amsalpha}
%   \bibliographystyle{alphaabbr}
%  % \bibliographystyle{alpha}
%  %\bibliography{/Users/seanprendiville/Dropbox/SeanBib.bib}
%  \bibliography{SeanBib.bib}
%  % \bibliography{SeanBib.bib}
% %\bibliography{PolylogBib} 
%}

\end{document}